\documentclass[12pt]{amsart}
\usepackage{amsmath}
\usepackage{amsthm}
\usepackage{amssymb}
\usepackage{amscd}
\usepackage{mathrsfs}
\usepackage{txfonts}
\newtheorem{theorem}{Theorem}[section]
\newtheorem{proposition}[theorem]{Proposition}

\newtheorem{corollary}[theorem]{Corollary}
\newtheorem{claim}[theorem]{Claim}

\theoremstyle{remark}
\newtheorem{remark}[theorem]{Remark}

\theoremstyle{definition}

\newtheorem{example}[theorem]{Example}
\begin{document}

\title{A note on Riley polynomials of $2$-bridge knots}

\author{Teruaki Kitano and Takayuki Morifuji}

\begin{abstract}
In this short note 
we show the existence of an epimorphism between groups of $2$-bridge knots by means of 
an elementary argument using the Riley polynomial. 
As a corollary, we give a classification 
of $2$-bridge knots 
by Riley polynomials. 
\end{abstract}

\thanks{2010 {\it Mathematics Subject Classification}.
57M25}

\thanks{{\it Key words and phrases.\/} 
Riley polynomial, $2$-bridge knot, epimorphism.}

\address{Department of Information Systems Science, 
Faculty of Science and Engineering, 
Soka University, 
Tangi-cho 1-236, 
Hachioji, Tokyo 192-8577, Japan}

\email{kitano@soka.ac.jp}

\address{Department of Mathematics, 
Hiyoshi Campus, 
Keio University, 
Yokohama 223-8521, Japan}

\email{morifuji@z8.keio.jp}
\maketitle

\section{Introduction}\label{section:1}

Let $K=S(\alpha,\,\beta)$ be a $2$-bridge knot given by the Schubert normal form 
(see \cite[Chapter~12]{BZH14-1}). 
Here $\alpha>0$ and $\beta$ are relatively prime odd integers 
satisfying $-\alpha<\beta<\alpha$. 
We denote the set of all such pairs $(\alpha,\,\beta)$ by $\mathcal{S}$. 
Two knots $S(\alpha,\,\beta)$ and $S(\alpha',\beta')$ are equivalent 
if and only if 
$\alpha=\alpha'$ and $\beta'\equiv \beta^{\pm1}\,(\text{mod}~ \alpha)$. 
In this note, we will identify 
$S(\alpha,\,\beta)$ with its mirror image $S(\alpha,\,\beta)^{*}=S(\alpha,-\beta)$. 
Namely 
we will consider two 
$2$-bridge knots $S(\alpha,\,\beta)$ 
and $S(\alpha',\,\beta')$ to be equivalent if and only if $\alpha=\alpha'$ and 
either $\beta'\equiv \beta^{\pm1}\,(\text{mod}~ \alpha)$ or 
$\beta'\equiv -\beta^{\pm1}\,(\text{mod}~\alpha)$. 
 
The knot group $G(K)=\pi_1(S^3\setminus K)$ of $K=S(\alpha,\,\beta)$ has a presentation 
\[
G(K)=\langle x,y\ |\ w x =y w\rangle
\]
where 
$w=x^{\epsilon_1}y^{\epsilon_2}\cdots x^{\epsilon_{\alpha-2}}y^{\epsilon_{\alpha-1}}$ and 
$\epsilon_i=(-1)^{\left[\frac{\beta}{\alpha}i\right]}$. 
Here 
we write $[r]$ for the greatest integer less than or equal to $r\in\Bbb{R}$. 
It is easily checked that 
$\epsilon_i=\epsilon_{\alpha-i}$ holds. 
The above presentation is not unique for a $2$-bridge knot $K$ itself, 
but the existence of at least one such presentation follows from 
Wirtinger's algorithm applied to the Schubert normal form of $S(\alpha,\,\beta)$. 
The generators $x$ and $y$ come from the two overpasses and present the 
meridian of $S(\alpha,\,\beta)$ up to conjugation. 

A representation 
$\rho:G(K)\rightarrow\mathit{SL}(2;\mathbb{C})$ 
is called {\it parabolic}\/ 
if $\mathrm{tr}\,\rho(x)=\mathrm{tr}\,\rho(y)=2$ holds and $\rho$ is nonabelian 
(i.e. $\mathrm{Im}~\rho$ is a nonabelian subgroup of $\mathit{SL}(2;\mathbb{C})$). 
We consider 
a map $\rho$ from $\{x,y\}$ to $\mathit{SL}(2;\mathbb{C})$ given by  
\[
x \mapsto 
X=\begin{pmatrix}
1 & 1\\
0 & 1\\
\end{pmatrix}\quad\mathrm{and}\quad
y \mapsto 
Y=\begin{pmatrix}
1 & 0\\
-u & 1\\
\end{pmatrix}, 
\]
where $u\not=0\in\mathbb{C}$. 
Moreover 
we define a matrix $W=(w_{ij})$ by 
$W=\prod_{l=1}^{(\alpha-1)/2}X^{\epsilon_{2l-1}}Y^{\epsilon_{2l}}$. 
Riley proved in \cite[Theorem~2]{Riley72} 
that the above map $\rho$ gives a parabolic representation of $G(K)$ if and only if 
$w_{11}=0$ holds. It is easy to see that any parabolic representation of $G(K)$ 
can be realized as the above form up to conjugation. 
We call this polynomial $w_{11}\in\mathbb{Z}[u]$ the {\it Riley polynomial}\/ 
of $K=S(\alpha,\,\beta)$ and denote it by $\phi_{S(\alpha,\,\beta)}(u)$. 
Then we have a map 
$\mathcal{R}:\mathcal{S}\to\mathbb{Z}[u]$ 
defined by the correspondence 
$(\alpha,\,\beta)\mapsto \phi_{S(\alpha,\,\beta)}(u)$. 

The purpose of the present note is to show the following result. 

\begin{theorem}\label{thm:cor}
Let $K_{1}=S(\alpha_1,\,\beta_1),\,K_{2}=S(\alpha_2,\,\beta_2)$ be $2$-bridge knots. 
If $\phi_{S(\alpha_1,\,\beta_1)}(u)=\phi_{S(\alpha_2,\,\beta_2)}(u)$, 
then there exists an isomorphism $G(K_{1})\to G(K_{2})$ 
and thus $K_{1}$ and $K_{2}$ are equivalent. 
\end{theorem}

\begin{remark}\label{rmk:converse}
As we will see in the next section, 
the Riley polynomial is not determined by the knot group itself 
(see Example~\ref{ex:K(7,3)}). 
Hence the converse statement of Theorem~\ref{thm:cor} 
does not hold. In other words, 
the map $\mathcal{R}:\mathcal{S}\to\mathbb{Z}[u]$ 
does not descend to the quotient $\mathcal{S}/$$\sim$, 
the set of equivalence classes of $2$-bridge knots. 
\end{remark}

The above theorem directly follows from Theorem~\ref{thm:main} which 
states the existence of an epimorphism between groups of $2$-bridge knots. 
In the next section, 
we quickly review some properties of the Riley polynomial. 
We will prove Theorem~\ref{thm:cor} in Section~\ref{section:3}. 
In the last section, 
we give a proof of a special case of Hartley-Murasugi's 
result to make this note as self-contained as possible. 

\section{Properties of Riley polynomials}\label{section:2}

The definition of the Riley polynomial $\phi_{S(\alpha,\,\beta)}(u)$ 
depends on a choice of a presentation of the knot group $G(K)$. 
Namely 
$G(K)$ itself does not determine $\phi_{S(\alpha,\,\beta)}(u)$. 

\begin{example}\label{ex:K(7,3)}
Let us consider a pair of equivalent $2$-bridge knots 
$K=S(7,3)$ and $K'=S(7,5)$. They have the isomorphic knot groups $G(K)\cong G(K')$ 
but different $\epsilon$-sequences 
$$
(\epsilon_i)=(1,1,-1,-1,1,1)\quad\text{and}\quad
(\epsilon_i')=(1,-1,1,1,-1,1).
$$
By a straightforward calculation, we have
$$ 
\phi_{S(7,3)}(u)=1-2u+u^2-u^3\quad\text{and}\quad
\phi_{S(7,5)}(u)=1-2u-3u^2-u^3.
$$
Namely it shows that $\phi_{S(7,3)}(u)\not=\phi_{S(7,5)}(u)$ 
although $K=S(7,3)$ and $K'=S(7,5)$ are equivalent. 
\end{example}

On the other hand, 
a $2$-bridge knot $S(\alpha,\,\beta)$ and its mirror image $S(\alpha,\,\beta)^*=S(\alpha,-\beta)$ 
share the Riley polynomial. 

\begin{claim}\label{claim:mirror}
Let $K=S(\alpha,\,\beta)$ and $K^*=S(\alpha,-\beta)$ its mirror image. 
Then we have $\phi_{S(\alpha,-\beta)}(u)=\phi_{S(\alpha,\,\beta)}(u)$. 
\end{claim}

\begin{proof}
Since $[r]+[-r]=-1$ holds for $r\in\Bbb{R}\setminus\Bbb{Z}$, 
$\epsilon$-sequences for $K$ and $K^*$ satisfy 
$\epsilon_i^*=-\epsilon_i$ for $1\leq i\leq \alpha-1$. 
By the inductive argument on the length of $\epsilon$-sequence, 
we can show that  
$$
w_{11}^*=w_{11},~
w_{12}^*=-w_{12},~
w_{21}^*=-w_{21},~
w_{22}^*=w_{22}
$$ 
for 
$W=(w_{ij})=\prod_{l=1}^{(\alpha-1)/2} X^{\epsilon_{2l-1}}Y^{\epsilon_{2l}}$ 
and 
$W^*=(w_{ij}^*)=\prod_{i=1}^{(\alpha-1)/2} X^{\epsilon_{2l-1}^*}Y^{\epsilon_{2l}^*}$. 
Therefore we have $\phi_{S(\alpha,-\beta)}(u)=\phi_{S(\alpha,\,\beta)}(u)$. 
\end{proof}

Let $\mathcal{S}_+$ be the subset of $\mathcal{S}$ satisfying $\beta>0$ 
and further 
$$
\mathcal{S}_+^*
=\left\{(\alpha,\,\beta)\in\mathcal{S}_+\,|\,\exists\,(\alpha,\,\beta')\in\mathcal{S}_+\, 
\text{s.\,t.}\ \beta\beta'\equiv1\, (\text{mod}\, \alpha)\, \text{and}\ \beta'<\beta
\right\}.
$$
We define $\overline{\mathcal{S}}$ to be $\mathcal{S}_+\setminus\mathcal{S}_+^*$. 
As an immediate corollary of Theorem~\ref{thm:cor} and Claim~\ref{claim:mirror}, 
we have the following. 

\begin{corollary}\label{cor:classification}
The restriction of $\mathcal{R}$, namely 
$\mathcal{R}|_{\overline{\mathcal{S}}}:\overline{\mathcal{S}}\to\mathbb{Z}[u]$ 
is injective. 
\end{corollary}

\section{Proof of Theorem~\ref{thm:cor}}\label{section:3}

Let $K_i=S(\alpha_i,\,\beta_i)~(i=1,2)$ 
be 2-bridge knots and $\phi_{S(\alpha_i,\,\beta_i)}(u)~(i=1,2)$ 
their Riley polynomials. 
To prove Theorem~\ref{thm:cor}, we first show the following. 

\begin{theorem}\label{thm:main}
If $\phi_{S(\alpha_2,\,\beta_2)}(u)$ is a factor of $\phi_{S(\alpha_1,\,\beta_1)}(u)$, 
then there exists an epimorphism from 
$G(K_1)$ to $G(K_2)$.
\end{theorem}

\begin{proof}
Now we fix the presentations of $G(K_1)$ and $G(K_2)$ as in Section~\ref{section:1}: 
$$
G(K_1)=\langle x_1,y_1~|~w_1 x_1=y_1 w_1\rangle,\quad
G(K_2)=\langle x_2,y_2~|~w_2 x_2=y_2 w_2\rangle.
$$
Assume $\phi_{S(\alpha_1,\,\beta_1)}(u)=\phi_{S(\alpha_2,\,\beta_2)}(u)\,\psi(u)$ for some $\psi(u)\in \mathbb{Z}[u]$, 
where 
the degree of $\phi_{S(\alpha_1,\,\beta_1)}(u)$ is $n+k$ and $\deg\phi_{S(\alpha_2,\,\beta_2)}(u)=n$. 
Let 
$u_1,u_2,\ldots, u_n$ be the zeros of $\phi_{S(\alpha_2,\,\beta_2)}(u)$ counting multiplicity. 
They give also zeros of $\phi_{S(\alpha_1,\,\beta_1)}(u)$ by the assumption. 

Since 
each zero of the Riley polynomial corresponds to 
a parabolic representation of the knot group, 
we can consider all parabolic representations 
$\rho_1^1,\rho_2^1,\ldots,\rho_{n+k}^1$ for $G(K_1)$ 
and $\rho_1^2,\rho_2^2,\ldots,\rho_{n}^2$ for $G(K_2)$, 
where  
$$
\rho_j^1(x_1)=\rho_j^2(x_2)
=\begin{pmatrix}
1 & 1\\
0 & 1
\end{pmatrix}\quad \text{and}\quad
\rho_j^1(y_1)=\rho_j^2(y_2)
=\begin{pmatrix}
1 & 0\\
-u_j & 1
\end{pmatrix}
$$ 
hold for $1\leq j \leq n$. 
We take the direct products of these representations as 
\begin{align*}
\Phi_1&:G(K_1)\ni g\mapsto \left(\rho_j^1(g)\right)_{j=1,\ldots,n+k}\in \mathit{SL}(2;{\mathbb C})^{n+k},\\
\Phi_2&:G(K_2)\ni g\mapsto \left(\rho_j^2(g)\right)_{j=1,\ldots,n}\in \mathit{SL}(2;{\mathbb C})^n
\end{align*}
and 
put $\Gamma_i=\text{Im}\,\Phi_i$ 
which is a subgroup of $\mathit{SL}(2;{\mathbb C})^{n+k}$ or $\mathit{SL}(2;{\mathbb C})^n$ 
generated by 
$\Phi_i(x_i)$ and $\Phi_i(y_i)$. 
The natural projection to the first $n$ factors 
$p:\mathit{SL}(2;{\mathbb C})^{n+k}\rightarrow  \mathit{SL}(2;{\mathbb C})^n$ 
induces a homomorphism 
$\overline{p}:\Gamma_1\rightarrow \mathit{SL}(2;{\mathbb C})^n$, 
and clearly 
$\overline{p}(\Phi_{1}(x_{1}))=\Phi_{2}(x_{2})$ and 
$\overline{p}(\Phi_{1}(y_{1}))=\Phi_{2}(y_{2})$ hold. 
Hence 
we obtain an epimorphism 
$\overline{p}:\Gamma_1\rightarrow \Gamma_2$. 

By Thurston's hyperbolization theorem for Haken $3$-manifolds, 
any knot in the $3$-sphere $S^3$ is either a torus knot $T(p,q)$, 
or a satellite knot, or a hyperbolic knot, 
i.e. its complement admits a complete hyperbolic metric with finite volume 
(see \cite{T97-1}). By \cite{Schubert54-1} a knot with bridge number $2$ 
cannot be a satellite, hence a $2$-bridge knot is a torus knot $T(2,q)$ where $q$ is odd 
(because the bridge number of $T(p,q)$ is $\min\{|p|,|q|\}$ by \cite{Schubert54-1}), 
or a hyperbolic knot. 
So let us consider the following two cases. 

(1) If $K_2$ is hyperbolic, 
there exists a parabolic and faithful representation of $G(K_2)$ into $\mathit{SL}(2;\mathbb{C})$ (see \cite{T97-1}). 
Hence $\Phi_2$ is injective. 
Applying the inverse of $\Phi_2$, 
we obtain an epimorphism 
\[
\Phi_2^{-1}|_{\Gamma_2}\circ\overline{p}\circ\Phi_1:
G(K_1)\rightarrow G(K_2).\]
This completes the proof in this case.

(2) If $K_2=T(2,q)$, a torus knot with an odd integer $q$, 
there is a parabolic faithful representation 
$G(K_2)/Z(G(K_2))\to \mathit{PSL}(2;\Bbb{R})\subset \mathit{PSL}(2;\Bbb{C})$, 
where $Z(G(K_2))\cong\Bbb{Z}$ is the center of $G(K_2)$ 
(see \cite{S83-1}). 
In fact, 
the image of $G(K_{2})/Z(G(K_{2}))$ is isomorphic to the 
$(2,q,\infty)$-triangle Fuchsian group in $\mathit{PSL}(2;\mathbb{R})$. 
This homomorphism lifts to $\mathit{SL}(2;\Bbb{C})$ and the lift is also faithful, 
so we can apply the same argument as in (1). 
Namely we have an epimorphism $G(K_1)\to G(K_2)/Z(G(K_2))$. 
Finally by a result of Hartley-Murasugi~\cite{HM78-1} 
(see also Section~\ref{section:appendix}), 
this epimorphism lifts to $G(K_2)$. 
\end{proof}

\noindent
{\it Proof of Theorem~\ref{thm:cor}}. 
If $\phi_{S(\alpha_1,\,\beta_1)}(u)=\phi_{S(\alpha_2,\,\beta_2)}(u)$, 
we obtain two epimorphisms 
$\varphi_{12}:G(K_1)\rightarrow G(K_2)$ 
and 
$\varphi_{21}:G(K_2)\rightarrow G(K_1)$ 
by Theorem~\ref{thm:main}. 
Hence we have the epimorphism $\varphi_{21}\circ\varphi_{12}:G(K_1)\to G(K_1)$. 
Since any knot group $G(K)$ is known to be \textit{Hopfian}\, 
(this follows from the facts that any knot group is residually finite \cite{Hempel87-1} 
and a finitely generated, residually finite group is Hopfian \cite{Malcev40-1}), 
namely, every epimorphism of $G(K)$ onto itself is an isomorphism, 
we see that $\varphi_{12}$ and $\varphi_{21}$ are injective. 
Therefore we can conclude $G(K_{1})\cong G(K_{2})$. 
Every $2$-bridge knot is prime, so $K_1$ is equivalent to $K_2$ 
(see \cite[Corollary~2.1]{GL89-1}). This completes the proof.

\section{Lifting problems}\label{section:appendix}

In this section, 
we prove that any epimorphism from any knot group onto the quotient of a torus knot group 
by its center can lift to the torus knot group. 
This is a special case of the result of Hartley and Murasugi (see \cite{HM78-1}), 
but here we give a proof to make this note self-contained. 

Let $K$ be a knot and $G(K)$ its knot group 
with the abelianization $\alpha:G(K)\rightarrow \mathbb{Z}\cong\langle m \rangle$. 
We simply write $G(p,q)$ to $G(T(p,q))$ for the $(p,q)$-torus knot $T(p,q)$. 
It is known that $G(p,q)$ has the following presentation
\[
G(p,q)=\langle x,y\,|\,x^p=y^q\rangle
\]
and its center is the infinite cyclic group generated by $z=x^p=y^q$. 
We write 
\[
\pi:G(p,q)\rightarrow \overline{G}(p,q)=G(p,q)/\langle z\rangle
=
\langle x,y\,|\,x^p=y^q=1\rangle
\]
and $
\gamma:\mathbb{Z}
=
\langle m \rangle\ni m\mapsto\overline{m}\in \mathbb{Z}/{pq}
=
\langle \overline{m} \,|\,\overline{m}^{pq}=1\rangle$. 
Further we take the abelianization 
$\beta:\overline{G}(p,q)\ni x\mapsto\overline{m}^q,~
y\mapsto\overline{m}^p\in\mathbb{Z}/pq$. 

\begin{proposition}[Hartley-Murasugi~\cite{HM78-1}]
If $\overline{\varphi}:G(K)\rightarrow \overline{G}(p,q)$ is an epimorphism 
such that $\beta\circ\overline{\varphi}=\gamma\circ\alpha$, 
then there exists a lift $\varphi:G(K)\rightarrow G(p,q)$ 
of $\overline{\varphi}:G(K)\rightarrow \overline{G}(p,q)$ 
such that $\pi\circ\varphi=\overline{\varphi}$.
\end{proposition}

\begin{proof}
The torus knot group 
$G(p,q)=\gamma^\ast(\overline{G}(p,q))$ can be described as the fiber product 
\[
\begin{CD}
\gamma^*(\overline{G}(p,q)) @>\widetilde{\gamma}~=~\pi>>  \overline{G}(p,q)\\
 @V\widetilde\beta VV       @VV\beta V\\
\mathbb{Z}@>>{\gamma}>\mathbb{Z}/{pq}.
\end{CD}
\] 
More precisely, 
the fiber product $\gamma^\ast(\overline{G}(p,q))$ is 
a subgroup of $\overline{G}(p,q)\times \mathbb{Z}$ as 
\[
\gamma^*\left(\overline{G}(p,q)\right)
=\left\{\left(\overline{g},m^s\right)\in 
\overline{G}(p,q)\times \mathbb{Z}\,|\,\beta\left(\overline{g}\right)=\gamma(m^s)
\right\}.
\]
Since there exists an exact sequence 
\[
1\rightarrow\mathrm{Ker}(\pi)=\langle (1,m^{pq})\rangle
\rightarrow
\gamma^*(\overline{G}(p,q))
\rightarrow
\overline{G}(p,q)
\rightarrow
1,
\]
we can see $G(p,q)\cong \gamma^*(\overline{G}(p,q))$. 

By the assumption, 
we have an epimorphism $\overline{\varphi}:G(K)\rightarrow \overline{G}(p,q)$ 
such that $\beta\circ\overline{\varphi}=\gamma\circ\alpha$. 
For any $g\in G(K)$ we then define 
\[
\varphi(g)=(\overline{\varphi}(g), \alpha(g))\in\overline{G}(p,q)\times\mathbb{Z},
\]
and 
it gives an epimorphism 
$\varphi:G(K)\rightarrow G(p,q)$ satisfying $\pi\circ\varphi=\overline{\varphi}$.
\end{proof}

\noindent
\textit{Acknowledgments}.\/
After finishing this work, 
Makoto Sakuma kindly told us the paper~\cite{Sakuma07} which shows the existence 
of an epimorphism between groups of $2$-bridge links from the view point of 
Markoff trace maps. 
This paper was written while the authors were visiting 
Aix-Marseille University. 
They would like to express their sincere thanks for the hospitality. 
The authors would like to thank Michel Boileau and Michael Heusener for helpful comments. 
This research was supported in part by JSPS KAKENHI  25400101 and 26400096. 


\end{document}